\documentclass[11pt]{amsart}
\usepackage{palatino}
\usepackage{amsmath, amsthm}
\usepackage{amssymb, graphicx, epsfig}

\usepackage[margin=1.5cm,
  vmargin={0pt,1cm},
  nohead]{geometry}
\newtheorem{theorem}{$\quad$Theorem}[section]

\newtheorem{corollary}[theorem]{$\quad$Corollary}

\newtheorem{lemma}[theorem]{$\quad$Lemma}

\newtheorem{definition}[theorem]{$\quad$Definition}

\newcommand{\R}{{\mathbb R}}
\newcommand{\C}{{\mathbb C}}
\newcommand{\Z}{{\mathbb Z}}

\theoremstyle{definition}
\newtheorem{remark}[theorem]{$\quad$Remark}

\newcounter{bibno}

\title{Flexible Isotopy Classification of Flexible Links}
\author{Johan Bj\"{o}rklund}

\begin{document}

\begin{abstract}
In this paper we define and study flexible links and flexible isotopy in $\R P^3\subset\C P^3$. Flexible links are meant to capture the topological properties of real algebraic links. We classify all flexible links up to flexible isotopy using Ekholm's interpretation of Viro's encomplexed writhe.
\end{abstract}

\maketitle

\section{Introduction}
Classical knot theory is the study of embeddings of $S^1$ into $\R^3$ or $S^3$ considered up to smooth isotopy. Real algebraic knot theory studies real algebraic curves in $\R P^3\subset\C P^3$ considered up to rigid isotopy, see for instance \cite{bjorklund}. In this paper we study smooth embeddings of surfaces into $\C P^3$ such that they have most of the topological properties of real algebraic curves.
Flexible curves in the plane were defined by Viro \cite{virocurves} following the observation that many of the properties of real algebraic curves of low degree in the projective plane only depended on certain topological properties.
The flexible curves were meant to caption these topological properties. It is natural to extend this definition to curves in projective 3-space to obtain flexible links.

\begin{definition}
An embedding $f:S_g\rightarrow \C P^3$, where $S_g$ is a closed connected orientable surface of genus $g$, is called a \emph{flexible link} of degree $d$ and genus $g$ if it fulfills the following three conditions:
\begin{itemize}
 \item $f(S_g)$ realizes $d[\C P^1]\in H_2(\C P^3)$.
 \item $f(S_g)$ is invariant under conjugation.
 \item $T_x(f(S_g))=\C T_x(\R f(S_g))$ for $x\in \R f(S_g)$, where $\R f(S_g)$ is the real part of $f(S_g)$ and $\C L$ denotes the complexification of the tangentline $L$, furthermore, we require the orientation induced by the complex structure on $\C P^3$ to agree with the orientation of $S_g$ along $\R f(S_g)$.
\end{itemize}
\end{definition}
We consider flexible links of genus $g$ up to automorphisms of $S_g$ respecting the real part. To simplify the notation we let $K=f(S_g)$ and we let $\R K:=K\cap \R P^3$.

\begin{remark}
The third condition differs slightly from Viro's definition for flexible curves in the plane. This is to simplify the definition of the encomplexed writhe (see Remark \ref{simplewrithe}).
\end{remark}

Viro's original definition also proscribed that the genus $g$ should satisfy $g=\frac{(d-1)(d-2)}{2}$. We omit this condition entirely since the genus of algebraic curves in space of a given degree can attain genus lower than this number without self-intersections.
Two flexible knots/links $K$ and $K'$ are said to be \emph{flexibly isotopic} if there exists a path in the moduli space of flexible knots/links connecting them, that is, if there exists a smooth isotopy taking $K$ to $K'$ such that it is at all times a flexible link. In the real algebraic world there is a similar notion of rigid isotopy. It is easy to see that a real algebraic knot/link is a flexible knot/link and that a rigid isotopy implies a flexible isotopy. To distinguish real algebraic knots up to rigid isotopy the main invariant is Viro's encomplexed writhe originally defined in \cite{virowrithe}. The encomplexed writhe is an invariant up to flexible isotopy using Ekholm's \cite{ekholmwrithe} definition of the encomplexed writhe in terms of shade numbers as seen in Section \ref{shade}.
Together with some other basic topological invariants and the smooth isotopy class of its real part the encomplexed writhe uniquely determines the flexible isotopy class of a flexible link as seen by the following theorem.
\begin{theorem}\label{main1}
Two flexible links $K$ and $K'$ are flexibly isotopic if and only if their degree, genus, Type and encomplexed writhe coincides and $\R K$ is smoothly isotopic to $\R K'$. 
\end{theorem}
Theorem \ref{main1} is proved in Section \ref{proofsection}. We will show how to construct flexible links in Section \ref{construct} with given properties subject to some basic restrictions given in Section \ref{basic}. Any flexible link must fulfill these restrictions, and given parameters

\section{Basic properties of flexible links}\label{basic}

In this section we define some basic properties of flexible links and show that some properties of real algebraic links survive to the flexible world. The proofs can usually be lifted directly from the corresponding proofs for real algebraic links.

\begin{definition}
We say that a flexible link $K$ is of \emph{Type II} if $K\setminus \R K$ is connected and of \emph{Type I} if $K\setminus \R K$ is not connected.
\end{definition}
\begin{remark}
It is easy to see that $K\setminus \R K$ can contain at most two parts by noting that for any connected component $A\subset K\setminus \R K$ we have that $conj(A)$ is also a component. When we travel along $K$ we change component when passing through $\R K$. This can only change our component to its conjugate, thus there can be no more than two components at most since $S_g$ is connected. 
\end{remark}
 
\begin{lemma}
The real part $\R K$ of a flexible link $K$ of genus $g$ has at most $g+1$ components, and if it has exactly $g+1$ components it must be of Type I.
\end{lemma}
\begin{proof}
This directly follows from the previous remark since removing more than $g+1$ circles from $K$ would result in more then two components.
\end{proof}

\begin{remark}
A flexible link without real components must be of Type II since $K\setminus \R K=K$.
\end{remark}
\begin{remark}
A flexible link $K$ is of type $I$ if and only if the quotient space $K/conj$ is orientable, since conjugation reverses orientation. A flexible link $K$ of type $I$ and genus $g$ where $\R K$ har $n$ components satisfies $g\equiv n+1$ counted modulo $2$ due to the symmetry properties of the conjugation. 
\end{remark}
\begin{lemma}
If $K$ is a flexible link of genus $g$ and type $I$ with $n$ real components, then $g\equiv n-1 \mod 2$.
\end{lemma}
\begin{proof}
Consider $K$ as a manifold in itself. After removing the $n$ real components $K\setminus \R K$ will consist of two connected components with $n$ removed discs. By the symmetry of the conjugation they both have the same genus $g´$. We reattach one of the real components to obtain a connected open submanifold. This results in a connected genus $2g'$ manifold with $2n-2$ discs removed. Each real component we reattach will increase the genus by one and remove $2$ ``holes''. Thus we have $2g'+n-1=g$ after all the real components have been reattached. Modulo $2$ this gives the relation stated in the lemma.  
\end{proof}
\begin{lemma}
If $K$ is a flexible link of degree $d$ then $\R K$ realizes $d\in \Z_2\simeq H_1(\R P^3)$.
\end{lemma}
\begin{proof}
We can calculate $d$ by counting the number of points (with signs) in the intersection of $K$ with the hyperplane at infinity (we can assume that it is generic after a real linear transformation of $\C P^3$). Non-real points come in complex conjugate pairs and so do not contribute to $d$ when calculated modulo $2$. The number of real points of at infinity $K$ is the same as the number of points in $\R K$ situated at infinity which determines its homology.  
\end{proof}

\begin{lemma}\label{isotopyextension}
Any smooth isotopy of the real part $\R K$ of a flexible link $K$ can be extended to a flexible isotopy such that its restriction to the real part is the smooth isotopy, and such that outside of a $\epsilon$ neighbourhood of $\R P^3$ it is constant.
\end{lemma}
\begin{proof}
In a small enough tubular neighbourhood of $\R P^3$, considered as a subset of $\C P^3$, the knot $K$ will be diffeomorphic to $\R K\times [-1,1]$ where the zerosection corresponds to the real part. On the interior subcylinder $\R K\times [-0.5,0.5]$ we extend the isotopy such that it respects the third property of flexible knots. On the remaining part, we are free to extend it however we wish as long as we avoid $\R P^3$ and respect the invariance under conjugation. We extend the isotopy along $\R K\times [0.5,1]$ in $\C P^3\setminus\R P^3$ such that it is constant along $\R K\times \{1\}$. The extension along the opposite end of the cylinder follows by conjugation. We can assume that there will be no selfintersections along the way since any such selfintersections would be from two complex conjugate points. But since we have an isotopy of a $2-$dimensional object in $6-$dimensional space we can assume that there are no such ``collisions'' along a generic isotopy. We now have a flexible isotopy of the small tubular neighourhood of $\R K$ in $K$ which is constant on the boundary and so trivially extended to a flexible isotopy of $K$ which is constant in the complement of this tubular neighbourhood.  
\end{proof}

\section{The encomplexed writhe}\label{shade}
In this section we recall the definition of Viro's encomplexed writhe, originally introduced in \cite{virowrithe}, using Ekholm's definition in terms of shade number from \cite{ekholmwrithe}. 
While Ekholm's definition holds true for more general settings (higher dimensions etc), we will restrict ourselves to flexible links in $\C P^3$. While the writhe was originally defined for real algebraic knots where it was considered as a rigid isotopy invariant, it is easy to see that it is well defined for flexible knots and also flexible isotopy invariant.
From an intuitive point of view there are 4 objects which we are interested in comparing, namely the flexible link $K$, its real parts $\R K$, the complex projective space $\C P^3$ and its real part $\R P^3$. The relationship between $K$ and $\C P^3$ is encoded in the degree $d$ of $K$. The relationship between $\R K$ and $\R P^3$ is determined by the smooth isotopy class of $\R K$. 
The writhe measures the relationship between $\R P^3$ and $K$.

\subsection{Defining the encomplexed writhe using shade classes}
Given a flexible link $K$ its encomplexed writhe $w(K)$ is calculated as follows.
Choose a nonzero section $N$ in the real normal bundle of $\R K$. Let $\R\hat K$ be $\R K$ pushed off along $\epsilon N$ for some small $\epsilon$. Let $\R w_N(K)=lk(\R K,\R \hat{K})$. Since $H_1(\R P^3)=\Z_2$ this is well defined (but dependent on the choice of $N$).
Now we take the normal section $\epsilon iN$ along $\R K$ when seen as a subset of $K$ and extend it to the entire $K$, such that it is zero in the complement of a small neighbourhood of $\R K$. Let $K'$ be $K$ pushed off along the resulting normal vectorfield. 
Due to the third condition of flexible knots, $K'$ will have no real points. Given a real point $x$ in $\R P^3$ we define its shade $\Gamma_x$ as the union of all the complexifications of real lines through $x$. The manifold $\Gamma_x$ is then a $4-$manifold. It is not orientable, however we can consider it to have twice $\R P^3$ as boundary, making it an orientable manifold.
We let $\C wr_N(K)=\frac{1}{2}K'\cdot \Gamma_x$. While $\Gamma_x$ depends on $x$ the intersection number does not.
The encomplexed writhe $w(K)$ is then defined as $w(K):=\C wr_N(K)+\R w_N(K)$.
\begin{remark}\label{simplewrithe}
The reason for the slight change in the definition of flexible link from Viro's planar definition is due to the pushoff along $iN$ defined above. By making sure that the complexification of the real tangent lines are the tangent planes to the entire flexible knot we ensure that $iN$ is still normal to $K$. 
\end{remark}

\section{Classifying flexible knots}\label{proofsection}
In this section we prove the main theorem of the paper, namely that writhe, degree, Type, genus and smooth topology of the real part is enough to determine the flexible isotopy class of a flexible link. Furthermore, we show that for any combination of these invariants (subject to some simple and necessary conditions) there exists a flexible link realizing them.
Before we prove the theorem, we show the following lemma:

\begin{lemma}\label{homoiso}
If $K$ is a flexible link and we have a homotopy $h$ respecting complex conjugation taking $K$ to a flexible link $K'$ such that $h$ is constant on a small neighbourhood $B$ of $\R P^3\subset\C P^3$ and such that the restriction of the homotopy to $K\setminus B$ is inside $\C P^3\setminus B$ then there exists a flexible isotopy taking $K$ to $K'$.
\end{lemma}
\begin{proof}
We can assume that the homotopy is generic. If it is not an isotopy then $h(p,t)=h(p',t)$ for some pair $p,p'$ with $p\neq p'$. If $\bar{p}=p'$ then $h(p,t)$ must be real contradicting the assumption. Thus $\bar{p}\neq p'$. Then we can extend any local homotopy to a homotopy respecting complex conjugation. Since $K$ is a two dimensional submanifold of $\C P^3$ a generic homotopy has no selfintersections thus we can assume that the homotopy can be modified around $p,t$ such that a the intersection is removed while respecting complex conjugation.  
\end{proof}

\begin{theorem}
Two flexible links $K$ and $K'$ are flexibly isotopic if and only if their degree, genus, Type and encomplexed writhe coincides and $\R K$ is smoothly isotopic to $\R K'$. 
\end{theorem}
\begin{proof}
It is obvious that if $K$ and $K'$ are flexibly isotopic then their degree, genus and encomplexed writhe coincides and $\R K$ is smoothly isotopic to $\R K'$. Our strategy will be to consider $K$ and $K'$ up to flexible isotopy and show that we can assume that they coincide on larger and larger parts. 
Assume that $K$ and $K'$ has coinciding degrees, genus, encomplexed writhe and that their real parts are smoothly isotopic. Due to Lemma \ref{isotopyextension} we can construct a flexible isotopy taking the $\R K$ to $\R K'$. From now on we may assume that $\R K=\R K'$.
Furthermore, we can assume that $K$ coincides with $K'$ on a small neighbourhood of $\R K=\R K'$. We wish to extend this neighbourhood on which the links coincide. The quotient space $K_C:=K/conj$ is a $2-$manifold with boundary. Each boundary component corresponds to a component of $\R K$. The manifold $K_C$ is orientable if and only if $K$ is of Type I. If $K_C$ is not orientable, we can (by the classification of 2-manifolds) consider it as a sphere with some disks removed and a number of crosscaps added. In each crosscap we can find a simple smooth connected curve $S$ whose preimage $\hat S$ under the projection $K\rightarrow K_C$ is also a simple closed connected curve which covers $S$ twice. The conjugation acts on $\hat S$ by rotating it by $\pi$ (using a diffeomorphism to $S^1$). We call such a curve $\hat S$ a \emph{dividing curve} in $K$.
Let $M$ be the number of crosscaps in $K_C$ and $N$ be the number of components of $\R K$. Then $M=0$ if $K$ is of Type I. If the links are of type II we have $M+N=g+1$ since the dividing curves together with the real components split $K$ into two parts diffeomorphic to spheres with $M+N$ disks removed. Since the number of crosscaps in $K_C$ only depend on the number of components of $\R K$, the Type and the genus we have that $K'_C\simeq K_C$.
Take two dividing curves $S\subset K$ and $S'\subset K'$. Let $p,\bar{p}$ be two points in $S$ such that they are sent to each other by the conjugation and let $p',\bar p'\in S'$. Let $B(p)$ be a small neighbourhood of $p$ in $K$. We can easily construct a flexible isotopy taking $p$ to $p'$ such that the isotopy is constant outside of $B(p)$ and $\bar{B(p)}=B(\bar p)$  by Lemma $\ref{homoiso}$. Since the flexible isotopy respects conjugation $\bar p$ will end up at $\bar p'$. We can thus assume that $p=p'$. Take a curve $c$ connecting $p$ to $\bar p$ in $S$ and a curve $c'$ connecting $p$ to $\bar p$ in $S'$. Since $\C P^3\setminus \R P^3$ is simply connected we can find a homotopy taking $c$ to $c'$ which can then be extended to a flexible isotopy such that $S$ is taken to $S'$.
We repeat this construction for each dividing curve arising from crosscaps in $K_C, K_C'$ choosing the paths such that $K$ and $K'$ coincide on a small neighbourhood $Q$ of the union of $\R K$ with $\cup S_i$ where $S_i$ are the dividing curves, and where the orientations on $Q$ induced by $K$ and $K'$ agree (the orientations on $K$ and $K'$ are induced from $\R K$ using the complex structure). The submanifold $Q$ divides $K$ and $K'$ into two pieces which are conjugate to each other. Since the orientations agreed around the dividing curves and around $\R K$, we have a diffeomorphism taking $K$ to $K'$ which is constant on $Q$ (then $K\setminus Q$ and $K'\setminus Q$ consist of two spheres with the same number of discs removed, the orientations agreeing ensures that we can "match up" the boundaries to the boundaries of $Q$ in a way that agrees). See Figure \ref{SQHpic} for a picture.

\begin{figure}[ht]
\includegraphics[scale=0.7]{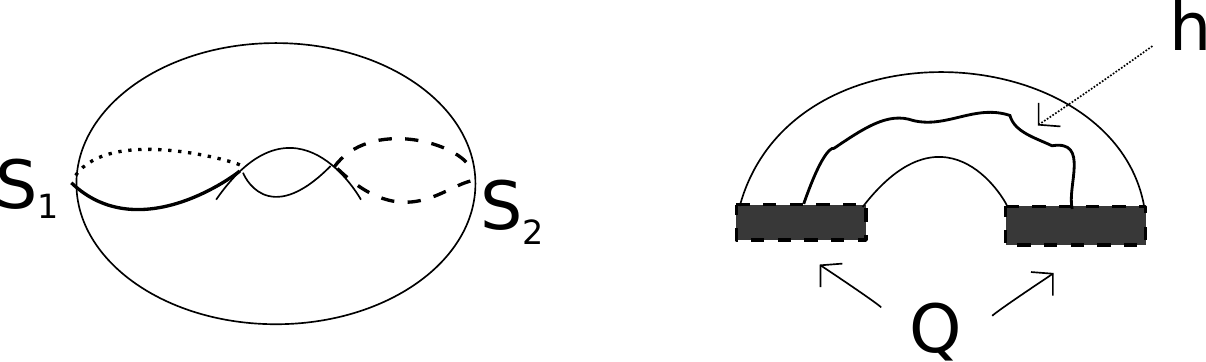}
\caption{A genus 1 flexible link $K$ with a real component $S_1$ and a dividing curve $S_2$, together with a picture of $K^+$ with $Q$ and $h$ marked.}
\label{SQHpic}
\end{figure} 

Choose a path $h$ representing a nontrivial class in $H_1(K,Q)$. Since $\C P^3\setminus \R P^3$ is simply connected we can deform this path together with a tubular neighbourhood however we wish. Choose a path $h'$ representing the same class in $H_1(K',Q)$ (using the diffeomorphism taking $K$ to $K'$ taking $Q$ to $Q$ identically). We then construct a homotopy taking  $h$ to $h'$ together with their tubular neighbourhoods and extend it (first locally and then by respecting conjugation) to a flexible isotopy of $K$ such that it leaves $Q$ invariant and takes $h$ to $h'$ together with a tubular neighbourhood. After repeating this strategy we can thus assume that $K\setminus Q\simeq K'\setminus Q$ consists of two complex conjugated connected components $D_K^+$, $D_K^-$ and $D_{K'}^+$,$D_{K'}^´-$ respectively, see Figure \ref{qlarge} for a picture.

\begin{figure}[ht]
\includegraphics[scale=0.4]{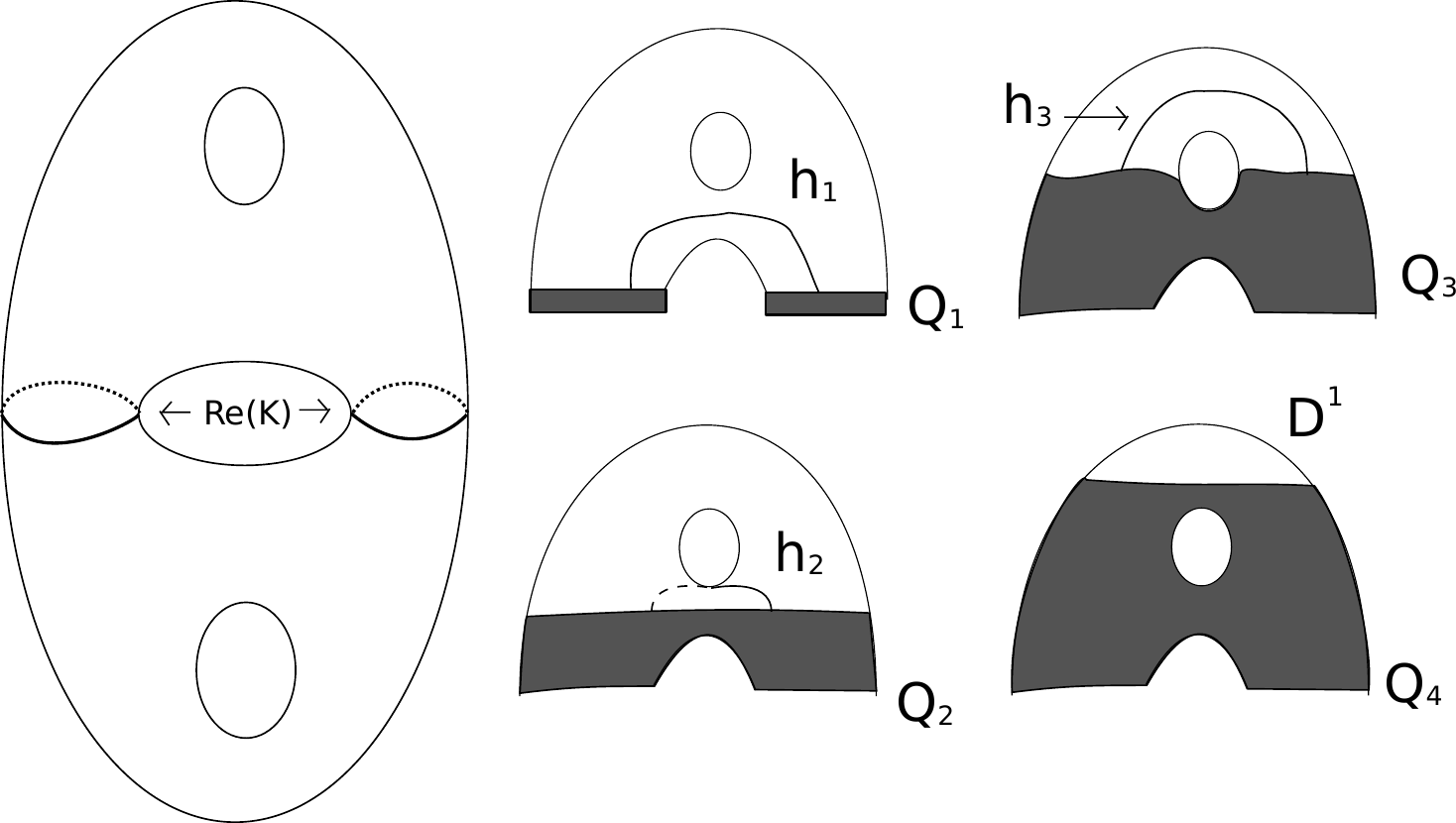}
\caption{We slowly enlarge the submanifold $Q$ of $K$ on which $K$ and $K'$ coincide, in this case for a Type 1 flexible link of genus $3$.}
\label{qlarge}
\end{figure} 

Let $h$ be the common boundary of $D_K^+$ and $D_{K'}^+$. Clearly $h$ is point-homotopic in $\C P^3\setminus \R P^3$ (since it bounded a disc). We contract $h$ to a point $p$ and examine which class in the homotopy group $\pi_2(\C P^3\setminus\R P^3,p)$ the discs belong to. If we can show that they belong to the same class then we are done, since we would have a homotopy taking $D_{K}^+$ to $D_{K'}^+$ which we would extend by complex conjugation so that it simultaneously takes $D_{K}^-$ to $D_{K'}^-$ which would imply that $K=K'$ when considering them up to flexible isotopy. Let $g$ represent $D_{K}^+$ and let $g'$ represent $D_{K'}^+$. By Hurewicz theorem we have that $\pi_2(\C P^3\setminus \R P^3)\simeq H_2(\C P^3\setminus \R P^3)$ since $\C P^3\setminus \R P^3$ is simply connected. The second homology $H_2(\C P^3\setminus \R P^3)$ is isomorphic to $\Z^2$ where the generator $(1,0)$ can be chosen such that it intersects infinity once with positive sign and the shade class $\Gamma$ zero times while the generator $(0,1)$ intersects infinity zero times and intersects the shade class $\Gamma$ once with positive sign.
Chose some real normal vector field $N$ along $\R K=\R K'$. Since the pushoffs of $\R K$ and $\R K'$ coincide along $N$ we must have that $\R w_N(K)=\R w_N(K')$. Since $wr(K)=wr(K´)$ we have that $\C wr_N(K)=\C wr_N(K')$. We recall that $\C wr_N(K)$ was calculated by taking the intersection number of $K$ pushed off along $iN$ with the shade class $\Gamma$. Since $K$ coincides with $K'$ outside of our discs their pushoffs coincide outside the discs and so the intersection number of $g$ with $\Gamma$ must coincide with the intersection number of $g'$ with $\Gamma$. For the same reason, the intersection numbers of $g,g'$ with infinity must also coincide. But then $g=g'$ in $H_2(\C P^3\setminus \R P^3)$ and so they coincide in homotopy as well, giving us that $K=K'$ when considered up to flexible isotopy.

\newpage
\section{Construction and complete classification of flexible links}\label{construct}
In this section we show how to construct flexible links with given Type, genus, degree, writhe and smooth isotopy class of the real part as long as they satisfy the restrictions in Section \ref{basic}. 
\begin{theorem}\label{thmconstruct}
Let $L$ be a smooth link of $n$ components in $\R P^3$ and let $d,g,w,X$ be integers such that 
\begin{itemize}
\item $g\geq0, X\in\{1,2\}$
\item $g\geq n+X-2$
\item $[L]=d[\R P^1]\in H_1(\R P^3)$ 
\item $1\leq g, g\equiv n+1 \mod 2$ if $X=1$
\item $w\equiv \frac{(d-1)(d-2)}{2}-g \mod 2$
\end{itemize}
Then there exists a flexible link $K$ of degree $d$, Type $X$ and genus $g$ such that $\R K$ is smoothly isotopic to $L$ and $wr(K)=w$.
\end{theorem}
\begin{proof}
It is always possible to find $n$ real algebraic rational curves $A_i$ such that the union of their real parts is smoothly isotopic to $L$ (if $L$ is empty we simply take a real rational curve without real points). We glue them together by moving $A_i$ and $A_{i+1}$ by a flexible isotopy until they intersect in a pair of nonreal complex conjugate points. At each such intersection we resolve the singularity locally to obtain a single flexible link $K_1$ with $\R K_1$ smoothly isotopic to $L$. The flexible link $K_1$ will have genus $n-1$. If $X=2$ we take $g-n+1$ rational curves without real points and attach them in the same way to make a new flexible link $K_2$ with genus $g$. If $X=1$ we take two nonintersecting, complex conjugate toruses without real points in $\C P^3$ and attach them in the same manner. This process continues until we have a flexible link $K_2$ with genus $g$. This process keeps the last equality at every step. Since $\R K_2$ is smoothly isotopic to $L$ and is a flexible curve of some degree $d'$, we must have that $[L]=d'[\R P^1]\in H_1(\R P^3)$, and so $2|d-d'$. Let $h$ be a complex line (without real points) and let $h'$ be its conjugate line. We attach this union to $K_2$ $d-d'$ times (counted with signs, that is , we can change their orientation). Since we attach a pair of spheres genus does not change. The resulting flexible link $K_3$ has degree $d$, genus $g$ and Type $X$. Choose a affine sphere $s$ in $\C P^3\setminus\R P^3$ intersecting $\Gamma$ twice, linking it with $\R P^3$. We can assume that $s$ does not intersect its complex conjugate $\bar s$. We attach $(wr(K)-wr(K_3))/2$ copies of the union of $s$ and $\hat s$ to $K_3$ (again counted with signs) to obtain a flexible link $K$ having all the requisite attributes.
\end{proof}
\begin{corollary} The restrictions in Theorem \ref{thmconstruct}, except the last one, are all necessary and come from Section \ref{basic}. Since there is just one flexible link up to flexible isotopy by Theorem \ref{main1} for a given list of such parameters this construction gives all flexible links up to flexible isotopy obeying this natural condition on the writhe. We plan to show that the last condition is necessary in an upcoming paper.
\end{corollary}
Note that this implies that there are comparatively many flexible knots compared to real algebraic knots, for instance, there are infinitely many pairwise non-isotopic flexible knots of degree $1$ while there is just one real algebraic knot (up to rigid isotopy). There exists no real algebraic knots in degree $0$ while there are (again infinitely many) flexible links there. 
For low degrees ($d\leq5$), the rigid isotopy class of real algebraic rational knots is uniquely determined by the encomplexed writhe (see \cite{bjorklund}). In degree $6$ there are examples of real algebraic rational knots with coinciding writhe which are not rigidly isotopic due to the real parts not being smoothly isotopic. It is still not known whether there exists two real algebraic rational knots of the same degree and with coinciding writhe and smoothly isotopic real parts which are not rigidly isotopic.

\end{proof}

\section*{Acknowledgements}
I would like to thank Tobias Ekholm and Grigory Mikhalkin for interesting discussions. The research was supported in part by TROPGEO grant of the ERC (the European Research Council) and the SNSF(the Swiss National Science Foundation) grant no. 140666 and 125070.

\end{document}